%% file: deg12.tex
\documentclass[12pt]{amsart}
\usepackage{hyperref}
\usepackage{amsmath}
\usepackage{amsfonts}
\usepackage{amssymb}
\usepackage{amsthm}
\usepackage{amsopn}
\usepackage{amscd}
\usepackage{enumerate}

\usepackage[all,cmtip]{xy}
\usepackage{pstricks}

\usepackage[utf8]{inputenc}
\usepackage{courier}
\usepackage{url}
\usepackage{color}
\usepackage{subfigure}
\usepackage{xypic}
\usepackage{graphicx}

\usepackage[vcentering,dvips]{geometry}
\title{Hilbert Schemes and Toric Degenerations for Low Degree Fano Threefolds}
\author{Jan Christophersen}
\address{Matematisk institutt, Postboks 1053 Blindern, University of Oslo, N-0216 Oslo, Norway}
\email{christoph@math.uio.no}
\author{Nathan Ilten}
\address{Department of Mathematics, University of California, Berkeley CA 94720}
\email{nilten@math.berkeley.edu}

\newcommand{\bp}{\mathbf{bp}}
\newcommand{\CC}{\mathbb{C}}
\newcommand{\Aff}{\mathbb{A}}
\newcommand{\ZZ}{\mathbb{Z}}
\newcommand{\QQ}{\mathbb{Q}}

\newcommand{\PP}{\mathbb{P}}

\newcommand{\K}{\mathcal{K}}
\newcommand{\mcL}{\mathcal{L}}
\newcommand{\T}{\mathcal{T}}
\newcommand{\tftw}{\mathfrak{tor}_{12}}
\newcommand{\tftv}{\mathfrak{tor}_{10}}
\newcommand{\tft}{\mathfrak{tor}}
\newcommand{\hilb}{\mathcal{H}}
\newcommand{\X}{\mathcal{X}}

\newcommand{\N}{\mathcal{N}}
\newcommand{\CO}{\mathcal{O}}
\newcommand{\I}{\mathcal{I}}

\newcommand{\numex}[2]{ {

  \vspace{.2cm}

\bf \noindent  #1}. 
#2

\vspace{.2cm}

}

\DeclareMathOperator{\Spec}{Spec}
\DeclareMathOperator{\Proj}{Proj}
\DeclareMathOperator{\Def}{Def}
\DeclareMathOperator{\Hom}{Hom}

\DeclareMathOperator{\TC}{TC}

\newtheorem{lemma}{Lemma}[section]
\newtheorem{prop}[lemma]{Proposition}

\newtheorem{thm}[lemma]{Theorem}

\newtheorem*{prob}{Problem}

\theoremstyle{definition}
\newtheorem*{ex}{Example}
\newtheorem*{remark}{Remark}

\include{figures12}

\newgray{gray1}{0.4}
\newgray{gray2}{0.5}
\newgray{gray3}{0.6}
\newgray{gray4}{0.7}
\newgray{gray5}{0.8}
\newgray{gray6}{0.9}
\newgray{gray7}{0.3}

\newcommand{\fooast}{$^*$}

\begin{document}
\maketitle
\begin{abstract}
For fixed degree $d\leq 12$, we study the Hilbert scheme of
degree $d$ smooth Fano threefolds in their anticanonical embeddings.
We use this to classify all possible degenerations of these varieties to toric Fano varieties with at most Gorenstein singularities.
\end{abstract}

\noindent Keywords: Fano threefolds, mirror symmetry, toric geometry,
Stanley-Reisner rings, Hilbert schemes \\

\noindent MSC: 14J45, 13F55, 14D15, 14M25
\section{Introduction}
Let $X\subset \PP^n$ be a scheme over $\CC$ and $\hilb_X$ the Hilbert scheme parametrizing subschemes of $\PP^n$ with the same Hilbert polynomial as $X$. Consider the following:
\begin{prob}
  Determine which irreducible components of $\hilb_X$ contain the point $[X]$ corresponding to $X$.
\end{prob}
\noindent The main result of the present paper is to solve this problem for the set of all \emph{Gorenstein toric Fano threefolds} of degree at most twelve, with respect to their anticanonical embeddings.
Our primary motivation for studying these Hilbert schemes comes from mirror symmetry and a conjectural relationship between toric degenerations and certain Laurent polynomials, see \S\ref{ss:mirror}.

We first fix notation and describe the relevant irreducible components of our Hilbert schemes in \S\ref{ss:irred}. In \S\ref{ss:main}, we present our main result, as well as our general strategy of proof.
In what follows, we will always be working over $\CC$.
A scheme is \emph{Fano} if it is projective, and the dualizing sheaf has a tensor multiple which is locally free and anti-ample.

\subsection{Irreducible Components}\label{ss:irred}
The majority of irreducible components we are interested in come from smooth Fano threefolds as follows. If $V$ is a  smooth Fano threefold with very ample anticanonical divisor, then the Hilbert polynomial of $V$ in its anticanonical embedding is determined solely by its degree $d=(-K_V)^3$. We denote the Hilbert scheme parametrizing subvarieties of $\PP(|-K_V|)$ with this Hilbert polynomial by $\hilb_d$.
The variety $V$ together with the anticanonical embedding corresponds to a point $[V]\in \hilb_d$, and this point lies on a single irreducible component, cf.~\cite{mm:cla}.

Deformation families of smooth Fano threefolds have been completely classified, see \cite{iskovskih:78a} and \cite{mori:81a}. Each family is distinguished by the degree $d$, the second and third Betti numbers $b_2$ and $b_3$, and the Lefschetz discriminant of any threefold in the family. For threefolds of degree less than $30$, the first three invariants suffice.
The degree must always be even, and can range from $2$ to $64$. 
We shall refer to the families and general elements of the families
interchangeably. 

Using this classification, we do some calculations to determine exactly which Fano threefolds of degree $d\leq 12$ have very ample anticanonical divisor and thus give rise to a component of $\hilb_d$, see \S\ref{sec:veryample}. The results are recorded in Table \ref{table:fanos}. Note that all varieties appearing in the table have index one, except for $V_8'$, which has index $2$. If $-K_V$ is very ample, we record how many global sections the corresponding normal sheaf has, which we calculate using our Proposition \ref{prop:h0N}. This is just the dimension of the corresponding component of $\hilb_d$, to which we also give a name in the table.

In addition to the components of Hilbert schemes corresponding to smooth Fano threefolds, we will encounter three non-smoothing components shown in Table \ref{table:gorfanos}. General elements of these components are certain Gorenstein trigonal Fano threefolds,  classified in \cite{cheltsov:05a}. We give a more precise description of these components in \S\ref{sec:exotic}. 

\begin{table}
\begin{center}
\begin{tabular}{|c| c| c| c| c| c|c| }
\hline
Name & Degree & $b_2$ & $b_3/2$ &  $-K_V$ very ample? & $h^0(\N)$ & Component of $\hilb_d$\\
\hline
$V_2$ & $2$ & $1$ & $52$  & No & N/A & N/A\\
$V_4$ & $4$ & $1$ & $30$  & Yes & 69 & $B_{69}^4$\\
$V_4'$ & $4$ & $2$ & $22$  & No & N/A& N/A\\
$V_6$ & $6$ & $1$ & $20$  & Yes & 69 & $B_{69}^6$\\
$V_6'$ & $6$ & $2$ & $20$  & No & N/A& N/A\\
$\PP^1\times S_1$ & $6$ &10 &0 & No & N/A&N/A\\
$V_8$ & $8$ & $1$ & $14$  & Yes & 75 & $B_{75}$\\
$V_8'$ & $8$ & $1$ & $21$  & No & N/A & N/A\\
$V_8''$ & $8$ & $2$ & $11$  & No & N/A & N/A\\
$V_{10}$ & $10$ & $1$ & $10$  & Yes & 85 & $B_{85}$\\
$V_{10}'$ & $10$ & $2$ & $10$  & Yes & 84 & $B_{84}$\\
$V_{12}$ & $12$ & $1$ & $7$  & Yes & 98 & $B_{98}$\\
$V_{12,2,6}$ & $12$ & $2$ & $6$  & Yes & 96 & $B_{96}$\\
$V_{12,2,9}$ & $12$ & $2$ & $9$  & Yes & 99 & $B_{99}$\\
$V_{12,3}$ & $12$ & $3$ & $8$  & Yes & 97 & $B_{97}$\\
$\PP^1\times S_2$ & $12$ &9 &0 & No & N/A&N/A\\
\hline
\end{tabular}
\end{center}\caption{Smooth Fano Threefolds of Low Degree}\label{table:fanos}
\end{table}

\begin{table}
\begin{center}
\begin{tabular}{|c| c| c| c| }
\hline
Name & Degree &  $h^0(\N)$ & Component of $\hilb_d$\\
\hline
$T_3$ & $10$ & $88$ & $B_{88}^\dagger$\\
$T_9$ & $10$ & $84$ &$B_{84}^\dagger$\\
$T_{25}$ & $12$& $99$ &$B_{99}^\dagger$\\
\hline
\end{tabular}
\end{center}\caption{Select Trigonal Fano Threefolds}\label{table:gorfanos}
\end{table}

\subsection{Main Result}\label{ss:main}
Let $\tft_d$ be the set of all degree $d$ toric Fano threefolds with at most Gorenstein singularities.
This set can be understood explicitly. Indeed, Gorenstein toric Fano varieties correspond to reflexive polytopes, see e.g. \cite[Theorem 8.3.4]{cox:11a}. The set of all three-dimensional reflexive polytopes has been classified by Kreuzer and Skarke in \cite{kreuzer}; the corresponding data may be found in the Graded Ring Database \cite{grdb}. 
As our main result, we determine exactly which components of $\hilb_d$ contain the points corresponding to elements of $\tft_d$.
This can be summarized as follows:

\begin{thm}\label{mainthm:0}
	For $d=4,6,8$ and any $X\in\tft_d$, $[X]$ is a smooth point of $\hilb_d$ on the same irreducible component  as $[V_d]$.
	For $d=10$ or $d=12$ and any $X\in\tft_d$, $[X]$ is a point on exactly the irreducible components of $\hilb_d$ recorded in Tables \ref{table:toricten} and \ref{table:torictwelve}, where we refer to $X$ by its number in the Graded Ring Database \cite{grdb}.
\end{thm}

\begin{table}
	\begin{center} {\small\torictentable}
\end{center}
	\caption{Degree ten toric Fano varieties and the scheme
	$\hilb_{10}$.}\label{table:toricten}
\end{table}

\begin{table}
	\begin{center} {\small\torictwelvetable}
\end{center}
	\caption{Degree twelve toric Fano varieties and the scheme
	$\hilb_{12}$.}\label{table:torictwelve}
\end{table}

We now outline the proof strategy:
\begin{enumerate}
  \item\label{item:one} For each $d\leq 12$, we identify a Stanley-Reisner scheme $S$ such that    
    $[S]\in\hilb_d$ is a smooth point lying on the component corresponding to $V_d$, cf.~Theorem \ref{thm:nice}.
  \item If $X$ is any Gorenstein toric Fano threefold of degree $4$, $6$, or $8$, then its moment polytope admits a certain ``good'' unimodular triangulation, cf.~Proposition \ref{prop:remain0}. This implies that $X$ has an embedded degeneration to the Stanley-Reisner scheme $S$ of \ref{item:one}. Hence, $[X]$ is a smooth point on the component corresponding to $V_d$. For many 
Gorenstein toric Fano threefold of degrees $10$ and $12$, a similar argument applies. 
  \item For $X$ one of the remaining degree $10$ and $12$ threefolds, we follow a general strategy described in \S\ref{sec:def2}. Using computer-assisted methods, we determine the number and dimensions of the tangent cone of $\hilb_d$ at $[X]$. These components are then matched with irreducible components of $\hilb_d$. There are three key ingredients:
    \begin{itemize}
      \item Comparison theorems allowing for effective computation of tangent and obstruction spaces of the Hilbert functor, see \S\ref{sec:def1};
      \item An understanding of the deformation theory of subvarieties of rational normal scrolls, see \S\ref{sec:rolling};
      \item A description of the universal family of $\hilb_{12}$ near a special point lying in the intersection of $B_{97}$, $B_{98}$, and $B_{99}$, see \S\ref{sec:bipyramid}.
    \end{itemize}
    The degree $10$ and $12$ cases are dealt with respectively in Theorems \ref{thm:re10} and \ref{thm:re12}.
\end{enumerate}

To the best of our knowledge, our strategy of studying the local structure of Hilbert schemes is new. We believe that these techniques will yield success in studying Hilbert schemes for Gorenstein toric Fano threefolds of higher degrees as well, although the increase in embedding dimension will lead to increased computational difficulties. 
We use a number of computer programs to carry out our calculations: {\small\verb+Macaulay2+} \cite{M2}, {\small\verb+Versal Deformations+} \cite{VD}, {\small\verb+TOPCOM+} \cite{TOPCOM}, and {\small\verb+4ti2+} \cite{4ti2}.
Supplementary material containing many of the computer calculations is available online \cite{supp}.

\subsection{Toric Degenerations and Connections to Mirror Symmetry}\label{ss:mirror}
A consequence of our main result is a classification of embedded degenerations of smooth Fano threefolds to Gorenstein toric Fano threefolds of degree at most twelve. Indeed, let $V$ be a general smooth Fano threefold of degree $d\leq 12$ with very ample anticanonical divisor, and $X\in \tft_d$. Then $V$ has an embedded degeneration to $X$ if and only if $[X]$ lies on the same irreducible component of $\hilb_d$ as $[V]$.

Toric degenerations are connected to mirror symmetry through the ansatz of \emph{extremal Laurent polynomials}, see \cite{przyjalkowski:09a}.
The quantum cohomology of a smooth Fano variety $V$ of dimension $n$ is conjecturally related to the Picard-Fuchs operator of a pencil $f\colon Y\to \CC$ called a (weak) Landau--Ginzburg model for $V$. The extremal Laurent polynomial ansatz conjectures that one should be able to take $Y=(\CC^*)^n$, that is, $f$ is a Laurent polynomial.
Furthermore, denoting the Newton polytope of $f$ by $\Delta_f$, it is expected that if $f$ gives a Landau--Ginzburg model for $V$, then $V$ degenerates to the toric variety whose moment polytope is dual to $\Delta_f^*$. Conversely, for any Fano toric variety $X$ with mild singularities smoothing to $V$, one expects to be able to find a Landau--Ginzburg model for $V$ in the form of a Laurent polynomial $f$ with $\Delta_f$ dual to the moment polytope of $X$.
T.~Coates, A.~Corti, S.~Galkin, V.~Golyshev,  and A.~Kasprzyk outline a program using these ideas to classify smooth Fano varieties in \cite{coates:fano}.

In \cite{przyjalkowski:09a}, V.~Przyjalkowski showed that for every smooth Fano threefold $V$ of Picard rank one, there is in fact a Laurent polynomial giving a weak Landau--Ginzburg model for $V$. Furthermore, in \cite{ilten:11b}, V.~Przyjalkowski, J.~Lewis, and the second author of the present paper showed that these Laurent polynomials are related to toric degenerations in the above sense. 
T.~Coates, A.~Corti, S.~Galkin, and A.~Kasprzyk have extended Przyjalkowski's result to Fano threefolds of higher Picard rank in \cite{coates:quantum}. Matching up our toric degenerations with the extremal Laurent polynomials they have found would provide evidence that the existence of an extremal Laurent polynomial is coupled to the existence of a toric degeneration.

Motivated by similar considerations, S.~Galkin classified all degenerations of smooth Fano threefolds to Fano toric varieties with at most \emph{terminal} Gorenstein singularities in \cite{galkin:07a}. This situation is however significantly different than the present one. Indeed, any Fano threefold with at most terminal Gorenstein singularities has a unique smoothing. This is no longer true if we relax the condition that the singularities be terminal; smoothings need not exist, and if they do, need not be unique.

\subsection*{Acknowledgements} 
We thank Cinzia Casagrande, Sergei Galkin, Jan Stevens, and the anonymous referee for helpful comments.

\section{Deformation Theory Methods}\label{sec:def}
In this section, we present some methods from deformation theory which will be necessary for our arguments. We first recall comparison theorems which will allow us to more easily compute the tangent and obstruction spaces of the Hilbert functor. Secondly, since many of the varieties we study occur as subvarieties of rational normal scrolls, we discuss the deformation theory of such varieties in terms of rolling factors. Thirdly, we outline a general strategy for identifying the component structure of Hilbert schemes at a given point; we use this strategy in \S\ref{sec:degten} and \S\ref{sec:toric}.

\subsection{Comparison Theorems and Forgetful Maps}\label{sec:def1}
Let $S = k[x_0,\dots, x_n]$, $A = S/I$ be a graded ring and $X = \Proj
A \subseteq \PP^n$. We may consider two deformation functors for $X$: the deformation
functor $\Def_X$ of isomorphism classes of deformations of $X$ as a
scheme, and the
local Hilbert functor $H_X$ of embedded deformations of $X$ in
$\PP^n$, see \cite{sernesi:06a} for details. The former has a formal semiuniversal element and the latter
a formal universal element. There is a natural forgetful map $H_X \to
\Def_X$. Let $T^1_X$ and $T^2_X$ be the tangent space and obstruction space for
$\Def_X$ and $T^i_{X/\PP^n}$, $i=1,2$,  the same for $H_X$. 

Assume now
that \emph{$A$ is Cohen-Macaulay of Krull dimension $4$} which is the case for
all schemes in this paper. We may use the comparison theorems of
Kleppe to relate the $T^i_{X/\PP^n}$ and $T^i_{X}$  to the degree 0 part of cotangent
modules of the algebra $A$. This has a large computational
benefit. For $H_X$, \cite[Theorem 3.6]{kleppe:79a} applied to our
situation yields 
\begin{gather*}
 {\Hom_A(I/I^2, A)}_0 =  {(T^1_{A/S})}_0  \simeq  T^1_{X/\PP^n} = H^0(X,\N_{X/\PP^n}) \\
 {(T^2_{A})}_0 \simeq  {(T^2_{A/S})}_0  \simeq T^2_{X/\PP^n} \, .
\end{gather*}
For $\Def_X$, \cite[Theorem 3.9]{kleppe:79a} applied to our situation yields
\begin{gather*}
 {(T^1_{A})}_0  \simeq T^1_{X} \\
0 \to  {(T^2_{A})}_0 \to  T^2_{X} \to  H^3(X, \CO_X)
\end{gather*}
where the latter sequence is exact. The cohomology $ H^3(X, \CO_X)$
does not appear in the statement in \cite{kleppe:79a}, but a careful
reading of the proof shows the existence of the sequence.

Thus if  $H^3(X, \CO_X) = 0$, as will be the case for the Fano
schemes in this paper, also $ {(T^2_{A})}_0 \simeq T^2_{X}$. The
Zariski-Jacobi sequence for $k \to S \to A$ reads 
$$\dots  \to T^1_{A/S} \to T^1_A \to T^1_S(A)  \to T^2_{A/S} \to
T^2_A \to T^2_S(A) \to \cdots$$
and since $S$ is regular $T^i_S(A) = 0$ for $i \ge 1$. This gives the
above written isomorphisms of obstruction spaces  $T^2_{A/S} \simeq T^2_A$ but also a
surjection   $T^1_{A/S} \to T^1_A$. By the above this means the
forgetful map $H_X \to \Def_X$ is surjective on tangent spaces and
injective on obstruction spaces, so it is smooth. 

The outcome of all this is that we may do versal deformation and local
Hilbert scheme computations using the vector spaces $(T^1_{A})_0$,
$(T^1_{A/S})_0$ and $(T^2_{A})_0$. Moreover, by smoothness of the
forgetful map, the equations
for the Hilbert scheme locally at $X$, in particular the component structure,
will be obstruction equations  for $\Def_X$ which involves much fewer parameters.

\subsection{Rolling Factors and Deformations}\label{sec:rolling}
When studying deformations of a scheme $X$, it is often useful to have a systematic way for writing down the equations in $X$. For subvarieties of rational normal scrolls, this is found in the method of \emph{rolling factors} introduced by Duncan Dicks, see e.g. \cite{dicks:88} and \cite[\S 1]{stevens:01a}. Since many of the Fano varieties we consider are subvarieties of scrolls, we summarize this method in the following.

Let $d_0\geq d_1\geq \ldots\geq d_k$ be non-negative integers and $d=\sum d_i$. Let $S$ be the image of 
$$\widetilde{S}=\PP\left(\bigoplus \CO_{\PP^1}(d_i)\right)$$
under the map defined by the twisting bundle $\CO(1)$. Then $\widetilde{S}$ is a $\PP^k$ bundle over $\PP^1$, and $S\subset \PP^{d+k}$ is cut out by the $2\times 2$ minors of
$$
M=\left (\begin{array}{c c c c c c}
	x_0^{(0)}&x_1^{(0)}&\cdots&x_{d_0-1}^{(0)}&x_0^{(1)}\cdots&x_{d_l-1}^{(l)}\\
	x_1^{(0)}&x_2^{(0)}&\cdots&x_{d_0}^{(0)}&x_1^{(1)}\cdots&x_{d_l}^{(l)}\\
\end{array}\right )	
$$
where $l$ is the largest integer such that $d_j\neq 0$.
We call $S$ a scroll of type $(d_0,d_1,\ldots,d_k)$. Note that $\widetilde{S}=S$ if and only if $d_k\neq 0$.

Let $f_0$ be a homogeneous polynomial in the variables $x_j^{(i)}$, $0\leq i \leq k$, $0 \leq j \leq d_i$ and suppose that every monomial in $f_0$ contains a factor from the top row of $M$. Then for every term in $f_0$, we may replace some $x_j^{(i)}$ by $x_{j+1}^{(i)}$ to obtain a new polynomial $f_1$. This process is called \emph{rolling factors}. Different choices of the factors might lead to different polynomials $f_1$, but any difference is contained in the ideal generated by the $2\times 2$ minors of $M$.

\begin{ex} Let $S$ be a scroll of type $(2,2,0,0)$ with corresponding matrix
\begin{align*}
	M=\left(\begin{array}{c c c c}
	x_0&x_1&y_0&y_1\\
	x_1&x_2&y_1&y_2
	\end{array}\right)
\end{align*}
in variables $x_0,x_1,x_2,y_0,y_1,y_2,z_1,z_2$.
Set  $f_0=x_0^2x_2-y_0z_1z_2$. By rolling factors we get the polynomial
$f_1=x_0x_1x_2-y_1z_1z_2$, whose factors we can again roll to get
$f_2=x_0x_2^2-y_2z_1z_2$.
\end{ex}

Let $f_0$ be a homogeneous polynomial as above of degree $e$, and suppose that we can subsequently roll factors $m$ times to get polynomials $f_0,\ldots,f_m$. Then the subvariety $X$ of $S$ cut out by the polynomials $f_0,\ldots,f_m$ is a divisor of type $eD-mF$, where $D$ is the hyperplane class  and $F$ is the image of the fiber class of $\widetilde{S}$. Furthermore, any such subvariety may be described in this matter.

Using this format for writing the equations of $X$, many of its deformations may be readily described, see \cite{stevens:01a}. Arbitrary perturbations of $f_0$ which still may be rolled $m$ times describe deformations of $X$ within its divisor class on $S$. Such deformations are called \emph{pure rolling factor} deformations. Perturbations of the entries of $M$ together with perturbations of the $f_i$ may give deformations of $X$ sitting on a deformed scroll. These deformations are called \emph{scrollar}. In general, there are also \emph{non-scrollar} deformations of $X$ which may not be described in either manner. For an illustration of all three types of deformations, see the example in \S\ref{sec:degten}.  

\subsection{Tangent Cones of Hilbert Schemes}\label{sec:def2}
Let $X$ be a subscheme of $\PP^n$. We would like to identify which components of the corresponding Hilbert scheme $\hilb$ the point $[X]$ lies on. In the following, we outline a general strategy for doing this.
	\begin{enumerate}
		\item Use obstruction calculus and the package {\small\verb+Versal Deformations+} to find the lowest order terms of obstruction equations for $X$. This will be feasible in the cases of interest to us due to the comparison theorems mentioned in \S\ref{sec:def1}. Let $Z$ denote the subscheme of the affine space $\Spec S^{\bullet} H^0(X,\N_{X/\PP^n})$ cut out by these equations; the tangent cone of $\hilb$ at $[X]$ is contained in $Z$.
		\item Do a primary decomposition of these lowest order terms to find the irreducible decomposition $Z_1,\ldots,Z_k$ of $Z$. Any component of the tangent cone $\TC_{[X]}\hilb$ is contained in some $Z_i$. Let $d_i$ denote the dimension of $Z_i$.
		\item For each $Z_i\subset Z$, find a tangent vector $v\in H^0(X,\N_{X/\PP^n})$ such that $v\in Z_i$ but $v\notin Z_j$ for $j\neq i$. Use {\small\verb+Versal Deformations+} to lift the first order deformation given by $v$ to higher order to get a one-parameter deformation $\pi:\X\to\Aff^1$ of $X$. In general, this may not be possible since the process of lifting to higher order may never terminate, resulting in a family defined by a power series. In practice however, for judicious choice of $v$, we almost always get a polynomial lifting after finitely many steps.
		\item We consider a general fiber $X'=\X_t$, $t\neq 0$ of $\X$. 
Suppose that $h^0(X',\N_{X'/\PP^n})=d_i$ and $T^2_{X'/\PP^n}=0$. Then $[X']$ lies on a component $B$ of $\hilb$ with $\dim B=d_i$. This implies that $Z_i$ is a component of $\TC_{[X]}\hilb$. Indeed, $[X]$ must also lie on $B$, and $\TC_{[X]}\hilb$ must have a component $Z_i'$ of dimension $d_i$ which contains $v$, since $X$ deforms to $X'$ with tangent direction $v$. Because $Z_i'$ contains $v$, it must be contained in $Z_i$, and equality follows from the equality in dimension. 
		\item Suppose that we have shown that $Z_i$ is a component of $\TC_{[X]}\hilb$ as described in step (iv). We now wish to determine for which component $B$ of $\hilb$  the tangent cone $\TC_{[X]} B$ contains $Z_i$. One approach is via deformation of the $X'$ above: if $X'$ deforms to some scheme $V$ for which we know $[V]$ lies on $B$, then $[X']$ lies on $B$ and $Z_i\subset \TC_{[X]} B$. A slightly more complicated approach is via degeneration of $X'$: suppose that $X'$ degenerates to a scheme $X_0$. If there is a component $B$ of $\hilb$ such that the degeneration direction from $X'$ to $X_0$, viewed as a deformation of $X_0$, only lies in $\TC_{[X_0]} B$ and no other components of $\TC_{[X_0]}\hilb$, then $[X']$ lies on $B$ and again $Z_i\subset \TC_{[X]} B$.
	\end{enumerate}
	Several difficulties may arise when attempting to put the above strategy into practice. For one, limits on computer memory and processor speed might make obstruction or primary decomposition calculations impossible. 
	Secondly, it could occur that the scheme $Z$ is not equal to $\TC_{[X]}\hilb$; this means that there will be some $Z_i$ which strictly contains a component of $\TC_{[X]}\hilb$. 
Thirdly, as mentioned in step (iii), lifting of one-parameter first order deformations might not terminate. 

For all the cases of present interest, these three problems almost never arise.
The only such problem we will encounter is in the few cases where some $Z_i$ is an embedded component. In these cases we can use deformation considerations to show that $Z_i$ does not correspond to a smoothing component of the Hilbert scheme. This is done in the final two examples of \S\ref{sec:toric}.

It can also occur in step (iv) that $h^0(X',\N_{X'/\PP^n})=d_i$ and $[X']$ is a smooth point of $\hilb$, but $T^2_{X'/\PP^n}\neq 0$. In such cases, an alternate strategy is needed to show that $[X']$ is indeed a smooth point of $\hilb$. One possible approach to deal with this problem is by using the structure of rolling factors, as we do for several cases in the proof of Theorem \ref{thm:re10}.

\section{Components of $\hilb_d$}\label{sec:comp}
In this section, we discuss the components of $\hilb_d$ which we shall encounter. In \S\ref{sec:dim}, we prove a general formula for the dimension of the Hilbert scheme component corresponding to a smooth Fano variety. In \S\ref{sec:compsmooth} we discuss those components corresponding to smooth Fano threefolds, and in \S\ref{sec:exotic} we discuss the non-smoothing components we encounter.

\subsection{Component Dimension}\label{sec:dim}
Before discussing specific components of the Hilbert schemes $\hilb_d$, we prove a result concerning Hilbert scheme component dimensions for Fano varieties in general. Recall from \S\ref{sec:def1} that given a  scheme $V \subset \PP^n$, the dimension of the tangent space of the Hilbert scheme at the corresponding point $[V]$ is just
 $h^0(\N_{V/\PP^n})$. For smooth Fano varieties, this can be computed as follows:
\begin{prop}\label{prop:h0N} Let $V\hookrightarrow \PP^n$ be a smooth
Fano variety in its anticanonical embedding and $\N_{V/\PP^n}$ the corresponding normal sheaf. Then
	$$
	h^0(\N_{V/\PP^n})=(n+1)^2-1-\chi(\Theta_V)
	$$
and $h^1(\N_{V/\PP^n})=0$, where $\Theta_V$ is the tangent sheaf of $V$.
	Furthermore, if $V$ is a threefold, then
	$$
	h^0(\N_{V/\PP^n})=g^2+3g+22-b_2+\frac{1}{2}b_3,
	$$
	where $g=\frac{1}{2}(-K_V)^3+1$ is the genus, and $b_2,b_3$
are the second and third Betti numbers.
\end{prop}
\begin{proof} By Kodaira vanishing, $h^1(\CO_V)=h^1(\CO_V(1))=0$. Kodaira vanishing also gives 
	$h^i(\Theta_V)=0$ for $i>1$ so
	$\chi(\Theta_V)=h^0(\Theta_V)-h^1(\Theta_V)$. Furthermore, it
follows from the Euler sequence that $h^1({{\Theta_{\PP^n}}}_{|V}
)=0$. The first claim then follows from the long exact cohomology
sequence coming from the normal sequence for $V$ in $\PP^n$.

Now assume that $\dim V=3$. By Hirzebruch-Riemann-Roch,
$$
\chi(\Theta_V)=\frac{1}{24}\deg(12c_1^3-19c_1c_2 + 12c_3)
$$
where the $c_i$ are the Chern classes of $\Theta_V$. We have that 
$\deg c_3=\chi_\mathrm{top}(V)=2+2b_2-b_3$ by Poincar\'e duality, and 
by definition of $g$, $\deg c_1^{3}=2g-2$.
Furthermore, an application of Hirzebruch-Riemann-Roch to $\CO_V$ gives $
\deg c_1c_2=24$.
Substituting these values into the above general dimension formula proves the second claim.
\end{proof}
\subsection{Smooth Fano Threefolds of Low Degree}\label{sec:veryample}\label{sec:compsmooth}
In Table \ref{table:fanos}, we list all families of smooth Fano threefolds of degree at most twelve. Degrees of general elements of these families and their topological invariants are taken from \cite{iskovskih:78a}, \cite{isk:80a}, and \cite{mori:81a}. Our names, referring both to the family and to general elements thereof, are non-standard. Below we calculate case by case whether the anticanonical divisor $-K_V$ is very ample. If so, we use Proposition \ref{prop:h0N} to calculate how many global sections the corresponding normal sheaf has. This gives us a list of all components of the Hilbert schemes $\hilb_d$ for $d\leq 12$ which correspond to smooth Fano threefolds, and the dimensions thereof. To summarize, we have the following:
\begin{prop}
The Hilbert schemes $\hilb_4$, $\hilb_6$, and $\hilb_8$ each have a single distinguished component, $\hilb_{10}$ has distinguished components $B_{84}$ and $B_{85}$ of dimensions $84$ and $85$, and $\hilb_{12}$ has distinguished components $B_{96}$, $B_{97}$, $B_{98}$, and $B_{99}$ of dimensions $96$, $97$, $98$, and $99$.
\end{prop}
\noindent The proof of this proposition follows immediately from Proposition \ref{prop:h0N} and the discussion in the remainder of this section.

\begin{lemma}\label{lemma:notva}
Let $V$ be a smooth Fano threefold with $\deg -K_V\leq 8$ and $-K_V$ very ample. Then $V$ is a complete intersection in its anticanonical embedding. In particular, $V_2$, $V_4'$, $V_6'$, $V_8'$, and $V_8''$ do not have very ample anticanonical divisor.
\end{lemma}
\begin{proof}
Let $g=\frac{1}{2}(-K_V)^3+1$. Then $V\subset \PP^{g+1}$ by Riemann-Roch. Thus $g>2$ for dimension reasons. If $g=3$, then $V$ must be a quartic hypersurface.
If $g=4$, then $V$ is the intersection of a cubic and a quadric, see \cite[Theorem 2.14]{cheltsov:05a}. This follows for example by considering the long exact sequence  coming from twists of 
\begin{equation*}
	0\to \I\to \CO_{\PP^{g+1}}\to \CO_V\to 0
\end{equation*}
by $\CO(2)$ and $\CO(3)$.
Finally, for the case $g=5$, it follows from \cite[Remark 1.9]{cheltsov:05a} that $V$ must be cut out by quadrics. For degree reasons, $V$ must thus be a complete intersection.

For the statement regarding $V_2$, $V_4'$, $V_6'$, $V_8'$,  and $V_8''$, note that, by the Lefschetz hyperplane Theorem, none of these varieties is a complete intersection in projective space. 
\end{proof}
The varieties $V_4$, $V_6$, and $V_8$ all have very ample anticanonical divisors. Indeed, in their anticanonical embeddings they are complete intersections in projective space of degrees $4$, $(2,3)$, and $(2,2,2)$. Furthermore, the varieties $V_{10}$ and $V_{12}$ have very ample anticanonical divisors, cf.~\cite{mukai:2004a}.
In its anticanonical embedding, $V_{10}$ is the intersection of the Grassmannian $G(2,5)$ in its Pl\"ucker embedding with a quadric and two hyperplanes. Likewise, $V_{12}$ is the intersection of the orthogonal Grassmannian $OG(5,10)$ in its Pl\"ucker embedding with seven hyperplanes.
We now deal with the remaining cases.

\begin{prop}\label{prop:caneq} Any smooth degree ten or twelve Fano threefold has very ample anticanonical divisor except for $\PP^1\times S_2$. Equations for 
their ideals in the anticanonical embedding are:
	\begin{enumerate}
\item 
In its anticanonical embedding in $\PP^7$, the ideal of of $V_{10}'$ is given by the minors of
\begin{align*}
M=	\left(\begin{array}{c c c c}
	x_0&y_0&z_0&w_0\\
	x_1&y_1&z_1&w_1
	\end{array}\right)
\end{align*}
together with cubics $f_0,f_1,f_2$, where $f_0$ is a general cubic which can be rolled twice to get $f_1$ and $f_2$.

			\item In its anticanonical embedding in
$\PP^8$, the ideal of $V_{12,2,6}$ is given by the minors of
\begin{align*}
	\left(\begin{array}{c c c c c}
		x_0&x_1&y_0&z_0&w_0\\ 
		x_1&x_2&y_1&z_1&w_1\\ 
	\end{array}\right)
\end{align*}
together with cubics $f_0,f_1,f_2,f_3$, where $f_0$ is a general cubic which can be rolled three times to get $f_1$, $f_2$, and $f_3$.

	\item In its anticanonical embedding in $\PP^8$, $V_{12,2,9}$
is defined by the $2\times 2$ minors of
			$$
			\left ( \begin{array}{c c c} u&x_1&y_0\\
y_1&v&x_2\\ x_0&y_2 &w\\
			\end{array}\right )
			$$
			together with a general quadric.
		\item In its anticanonical embedding in $\PP^8$,
$V_{12,3}$ is defined by the $2\times 2$ minors of the two matrices
\begin{align*}
	\left(
	\begin{array}{c c c c}
	x_{000}&x_{100}&x_{001}&x_{101}\\
	x_{010}&x_{110}&x_{011}&x_{111}\\
	\end{array}\right)
	\qquad \left(
	\begin{array}{c c c c}
	x_{000}&x_{010}&x_{001}&x_{011}\\
	x_{100}&x_{110}&x_{101}&x_{111}\\
	\end{array}\right)
			\end{align*} in
$\CC[x_{ijk},t]_{i,j,k\in\{0,1\}}$ along with a general quadric.
	\end{enumerate}
\end{prop}
\begin{proof} 
The variety  $\PP^1\times S_2$ is the product of $\PP^1$ with a del Pezzo surface $S_2$ of degree two, and it is well-known that the latter does not have very ample anticanonical divisor.
	
	For $V_{10}'$, note that the variety $V$ described by the equations in (i) is a divisor of type $3D-2F$ on a scroll $S$ of type $(1,1,1,1)$. Since this divisor is basepoint free, $V$ is smooth. The divisor $-K_S$ is equivalent to $4D-2F$ (see e.g. \cite[pp. 23]{kollar:00a}), and the adjunction formula thus shows that $-K_V=\CO_V(1)$. Since $V$ is not cut out by quadrics, it cannot be $V_{10}$ and must thus be $V_{10}'$.

For $V_{12,2,6}$, note that the variety $V$ described by the equations in (ii) is a divisor of type $3D-3F$ on a scroll $S$ of type $(2,1,1,1)$. Since this divisor is basepoint free, $V$ is smooth. The divisor $-K_S$ is equivalent to $4D-3F$, and the adjunction formula thus shows that $-K_V=\CO_V(1)$. Since $V$ is not cut out by quadrics, it cannot be $V_{12}$ or one of the other two degree twelve threefolds dealt with below. 

For the remaining cases, we use the descriptions for the Fano threefolds found in
\cite{mori:81a}. 
The threefold $V_{12,2,9}$ is a divisor of bidegree $(2,2)$ in $\PP^2\times
\PP^2$. The equations for the Segre embedding of $\PP^2\times \PP^2$
in $\PP^8$ are given by the $2\times 2$ minors of a general $3\times
3$ matrix. In this embedding, a divisor of bidegree $(2,2)$ is
given by a quadric. It follows from the adjunction formula that this
is in fact the anticanonical embedding, so in particular $-K_{V_{12,2,9}}$ is very ample.

The threefold $V_{12,3}$ is a double cover of $\PP^1\times\PP^1\times\PP^1$ with
branch locus a divisor of tridegree $(2,2,2)$.  
Now, the equations listed in (iv) describe a smooth variety $V$ which is the cone over $\PP^1\times\PP^1\times\PP^1$ embedded via $\CO(1,1,1)$, intersected with a general quadric. Projection from the point $\{x_{ijk}=0,t=1\}$ gives a map $V\to\PP^1\times\PP^1\times\PP^1$ whose ramification locus is a divisor of type $(2,2,2)$. Thus, $V$ equals $V_{12,3}$. To check that this is in fact its anticanonical embedding, we again use adjunction: a straightforward toric calculation shows that the anticanonical divisor on the cone over $\PP^1\times\PP^1\times\PP^1$ is the pullback of $\CO(3)$, so intersection with a quadric gives that $-K_V$ is the pullback of $\CO(1)$.
\end{proof}
\begin{remark}
  All components $B$ of $\hilb_d$ for $d\leq 12$ corresponding to smooth Fano threefolds are unirational, that is, there is some dominant rational map $\Aff^k\dashrightarrow B$. Indeed, Mori and Mukai show that the variety parametrizing any family of smooth Fano threefolds is unirational \cite{mm:cla}.  
Since the map $H_X\to \Def_X$ is smooth for any Fano threefold $X$, it follows that the corresponding Hilbert scheme component is also unirational.
\end{remark}
\subsection{Non-smoothing Components of $\hilb_d$}\label{sec:exotic}
In our study of $\hilb_{10}$ and $\hilb_{12}$, we will encounter three additional components which do not correspond to smooth Fano threefolds, but instead non-smoothable trigonal Fano threefolds with Gorenstein singularities, see \cite{cheltsov:05a}.

We first describe an $88$-dimensional component $B_{88}^\dagger$ of $\hilb_{10}$. Consider the 
matrix
$$M=\left(
\begin{array}{c c c}
x_0&y_2&y_1\\
y_2&x_1&y_0\\
y_1&y_0&x_2\\
\end{array}\right).
	$$
Let $g_0,g_1,g_2$ be general quadrics in $x_i,y_j,z_1,z_2$, and let $f_0,f_1,f_2$ be the cubics defined by
$$
M\cdot\left(\begin{array}{c} g_0\\ g_1\\ g_2\end{array}\right)=\left(\begin{array}{c} f_0\\ f_1\\ f_2\end{array}\right).
$$
Note that $f_1$ and $f_2$ have been constructed from $f_0$ in a manner similar to rolling factors.

Let $I$ be the ideal generated by the $2\times 2$ minors of $M$ and $f_0,f_1,f_2$.
This cuts out a singular degree $10$ Fano variety $V\subset\PP^7$ corresponding to a point $[V]\in\hilb_{10}$. Indeed, this is the case $T_{3}$ of \cite{cheltsov:05a}.
Using {\small\verb+Macaulay2+}, we compute that $h^0(V,\N)=88$, and that all deformations of $V$ come from perturbing the quadrics $g_0,g_1,g_2$ and are unobstructed. Thus, $[V]$ is a smooth point on an $88$-dimensional component $B_{88}^\dagger$  of $\hilb_{10}$.

The remaining two non-smoothing components may be nicely described using rolling factors.
The Hilbert scheme $\hilb_{10}$ has an additional $84$-dimensional component $B_{84}^\dagger$ which is \emph{not} the component $B_{84}$.
Consider the matrix  
$$M=\left(
\begin{array}{c c c c c}
x_0&x_1&y_0&y_1\\
x_1&x_2&y_1&y_2\\
\end{array}\right).
	$$
With additional variables $z_1,z_2$, its maximal minors define a scroll of type $(2,2,0,0)$. Let $f_0$ be a general cubic which can be rolled $2$ times to $f_1$ and $f_2$. The ideal generated by the minors of $M$ together with these three cubics cuts out out a singular degree $10$ Fano variety $V\subset\PP^7$ corresponding to a point $[V]\in\hilb_{10}$. Indeed, this is the case $T_{9}$ of \cite{cheltsov:05a}.

Using {\small\verb+Macaulay2+}, we compute that $h^0(V,\N)=84$, and that all deformations of $V$ are of pure rolling factor type. Thus, $[V]$ is a smooth point on a $84$-dimensional component of $\hilb_{10}$, and $V$ cannot be smoothed. It follows that the component $B_{84}^\dagger$ of $\hilb_{10}$ upon which $[V]$ lies is not $B_{84}$.

The Hilbert scheme $\hilb_{12}$ has an additional $99$-dimensional component $B_{99}^\dagger$ which is \emph{not} the component $B_{99}$.
Consider the matrix  
$$M=\left(
\begin{array}{c c c c c}
x_0&x_1&x_2&x_3&y_0\\
x_1&x_2&x_3&x_4&y_1\\
\end{array}\right).
	$$
With additional variables $z_1,z_2$, its maximal minors define a scroll of type $(4,1,0,0)$. Let $f_0$ be a general cubic which can be rolled $3$ times to $f_1,f_2$, and $f_3$. The ideal generated by the minors of $M$ together with these four cubics cuts out out a singular degree $12$ Fano variety $V\subset\PP^8$ corresponding to a point $[V]\in\hilb_{12}$. Indeed, this is the case $T_{25}$ of \cite{cheltsov:05a}.

Using {\small\verb+Macaulay2+}, we compute that $h^0(V,\N)=99$, and that the obstruction space $T_V^2$ vanishes. Thus, $[V]$ is a smooth point on a $99$-dimensional component of $\hilb_{12}$. 
All deformations of $V$ are of pure rolling factor type, so $V$ cannot be smoothed. Thus, the component $B_{99}^\dagger$ of $\hilb_{12}$ upon which $[V]$ lies is not $B_{99}$.

\begin{remark}
	It follows from the description in  \cite[Theorem 1.6]{cheltsov:05a} of Gorenstein trigonal Fano threefolds that each family is parametrized by a unirational variety. By arguments similar to in the previous section, one can show that $B_{84}^\dagger$, $B_{88}^\dagger$, and $B_{99}^\dagger$ are also unirational.
\end{remark}

\begin{remark}
Not every type of singular trigonal Fano  described in \cite[Theorem 1.6]{cheltsov:05a} describes a new Hilbert scheme component. For example, a routine calculation shows that the case of $T_7$ (a scroll of type $(2,1,1,0)$ and a cubic rolled twice) always has a non-scrollar deformation which deforms it to $V_{10}$.
On the other hand, there may be additional components of $\hilb_{10}$ and $\hilb_{12}$ not discussed here whose general elements are singular Fano varieties.
\end{remark}

\section{Stanley-Reisner Schemes and Degenerations}\label{sec:sr} 
In this section, we first recall some basic facts about simplicial
complexes and Stanley-Reisner schemes, see for example
\cite{stanley:83a}.
We will then discuss degenerations to particular Stanley-Reisner schemes.
This will be fundamental to our study of $\hilb_d$.
\subsection{Stanley-Reisner Basics}
Let $[n]$ be the set $\{0,\ldots,n\}$ and 
$\Delta_n$ be the full simplex $2^{[n]}$. An abstract
\emph{simplicial complex} is any subset $\K\subset\Delta_n$ such that
if $f\in\K$ and $g\subset f$, then $g\in \K$. Elements $f\in\K$ are
called \emph{faces}; the dimension of a face $f$ is $\dim
f:=\#f-1$. Zero-dimensional faces are called \emph{vertices};
one-dimensional faces are called \emph{edges}. The \emph{valency} of a
vertex is the number of edges containing it. 
Two 
simplicial complexes are isomorphic if there is a bijection of the
vertices inducing a bijection of all faces. We will not differentiate
between isomorphic complexes.

Given two simplicial complexes $\K$ and $\mcL$, their \emph{join} is
the simplicial complex
$$
\K * \mcL=\{f\vee g\ | \ f\in\K,\ g\in\mcL\},
$$
where $f\vee g$ denotes the disjoint union of $f$ and $g$.
To any simplicial complex $\K\subset\Delta_n$, we associate a
square-free monomial ideal $I_\K\subset \CC[x_0,\ldots,x_n]$
$$
I_\K:=\langle x_p \ | \ p\in\Delta_n\setminus\K\rangle
$$
where for $p\in\Delta_n$, $x_p:=\prod_{i\in p}x_i$. This gives rise to
the \emph{Stanley-Reisner ring} $A_\K:=\CC[x_0,\ldots,x_n]/I_\K$ and a
corresponding projective scheme $\PP(\K):=\Proj A_\K$ which we call a
Stanley-Reisner scheme.  Certain properties of the scheme $\PP(\K)$ are reflected in the combinatorics of the
complex $\K$. For example, each face $f\in\K$ corresponds to some $\PP^{\dim
f}\subset \PP(\K)$ and the intersection relations among these projective
spaces are identical to those of the faces of $\K$. In particular,
maximal faces of $\K$ correspond to the irreducible components of $X$.

In this paper, we will only consider Stanley-Reisner schemes of the form $\PP(\K*\Delta_0)$, where $\K$ is topologically the triangulation of a two-sphere. Such schemes are Gorenstein Fano threefolds, and are embedded via the anticanonical divisor, see \cite[Proposition 2.1]{paperone}.
\subsection{Degenerations to Stanley-Reisner Schemes}
We recall the correspondence between unimodular triangulations and degenerations of toric varieties. Consider some lattice $M$ and some lattice polytope
$\nabla\subset M_\QQ$ in the associated $\QQ$-vector space. By
$\PP(\nabla)$ we denote the toric variety
$$
\PP(\nabla)=\Proj \CC[S_\nabla]
$$
where $S_\nabla$ is the semigroup in $M\times \ZZ$ generated by the
elements $(u,1)$, $u\in \nabla \cap M$. By Theorem 8.3 and Corollary
8.9 of \cite{sturmfels:96a}, square-free initial ideals of the toric
ideal of $\PP(\nabla)$ are exactly the Stanley-Reisner ideals of
unimodular regular triangulations of $\nabla$, see loc.~cit.~for definitions.

We now describe the triangulated two-spheres we will need. 
First of all, let $T_4=\partial \Delta_3$, $T_5=(\partial \Delta_2)*(\partial \Delta_1)$, and for $6\leq i \leq 10$, let $T_i$ be the unique triangulation of the sphere with $i$ vertices having valencies four and five.\footnote{The $T_i$ arise naturally as the boundary complexes of the convex deltahedra (excluding the icosahedron).}
For concrete realizations of these triangulations, see \cite[Figure 1]{paperone}. The corresponding Stanley-Reisner schemes $\PP(T_i*\Delta_0)$ satisfy some nice properties:
\begin{thm}[{See \cite[\S 3]{paperone}}]\label{thm:nice}
	Let $4\leq i \leq 10$ and $d=2i-4$ and let $V_d$ be a general rank one index one degree $d$ smooth Fano threefold. Then $V_{d}$ degenerates to $\PP(T_i*\Delta_0)$ in its anticanonical embedding. Furthermore, $[\PP(T_i*\Delta_0)]\in\hilb_d$ is a smooth point.
\end{thm}

This theorem alone allows us to determine the position of toric Fano threefolds of degree $d$ with at most Gorenstein singularities in $\hilb_d$ for $2<d< 10$:

\begin{prop}\label{prop:remain0}
	For $d=4,6,8$, let $X$ be a toric Fano threefold of degree $d$ with at most Gorenstein singularities. Then $[X]$ is a smooth point on the component of $\hilb_d$ corresponding to $V_d$. 
In particular, $X$ always admits an embedded smoothing to a smooth Fano threefold.
\end{prop}
\begin{proof}
  Using the classification of such varieties by \cite{kreuzer} and computer calculations with {\small\verb+TOPCOM+}, we verify that for any such variety $X$ of degree $d=4,6$, its moment polytope 
has a regular unimodular triangulation of the form $T_i*\Delta_0$, $i=4,5$. Thus, $X$ degenerates to $\PP(T_i*\Delta_0)$, so $[X]$ is a smooth point of $\hilb_d$ on the same component as  $[V_d]$.

For $d=8$, similar calculations show that the moment polytopes have regular  unimodular triangulations of the form $T_6*\Delta_0$ or of the form $T_6'*\Delta_0$, where $T_6'$ is the unique triangulation of the two-sphere with valencies $3,3,4,4,5,5$. For any $X$ with a triangulation of the former sort, as above $[X]$ is a smooth point of $\hilb_8$ on the same component as $[V_8]$. On the other hand, $[\PP(T_6'*\Delta_0)]$ is a smooth point of $\hilb_8$ by \cite[Corollary 2.5]{ishida:81a}. Since there are polytopes admitting both types of triangulations, this point must thus lie on the same component as $[V_8]$. Thus, varieties $X$ with triangulations of the latter sort also are smooth points of $\hilb_8$ on the same component as $[V_8]$.
\end{proof}

\begin{ex}
There is a single toric Fano threefold of degree $4$ with Gorenstein singularities, which is cut out by the quartic $x_1x_2x_3x_4-x_0^4$. A degeneration to $\PP(T_4*\Delta_0)$ is given by degenerating the quartic to its first term.
\end{ex}

We will deal with the degree $10$ case in the following section. For the degree $12$ case, we will need an additional simplicial complex. Let $T_8'$, the bipyramid over the hexagon, be the unique triangulation of the sphere with valencies $4,4,4,4,4,4,6,6$. This triangulation is pictured in Figure \ref{fig:tri}, where the hollow dot represents the point at infinity.
\begin{figure}
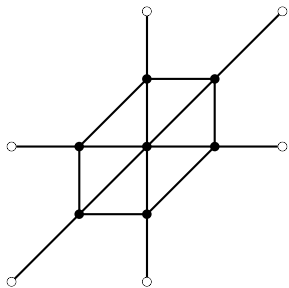
 
	\begin{center}
\tritwelve
\end{center}
\caption{The triangulation $T_8'$}\label{fig:tri}
\end{figure}

The Stanley-Reisner scheme corresponding to this triangulation also arises as a degeneration of smooth Fano threefolds: 

\begin{prop}\label{prop:srdegen} 
		The smooth Fano threefolds $V_{12}$, $V_{12,2,9}$ and $V_{12,3}$ all degenerate to $\PP(T_8'*\Delta_0)$.
\end{prop}
\begin{proof} 
	First of all, note that the polytope dual to number 127896 from \cite{grdb} has  regular unimodular triangulations to both $T_8*\Delta_0$ and $T_8'*\Delta_0$. Thus, the corresponding toric variety degenerates to both  $\PP(T_8*\Delta_0)$ and $\PP(T_8'*\Delta_0)$. But since $[\PP(T_8*\Delta_0)]$ is a smooth point of the Hilbert scheme,  $[\PP(T_8*\Delta_0)]$ and $[\PP(T_8'*\Delta_0)]$ must lie on the same component, and since $V_{12}$ degenerates to $\PP(T_8*\Delta_0)$, it also degenerates to $\PP(T_8'*\Delta_0)$.

	For the remaining degenerations, we use the equations from Proposition \ref{prop:caneq} and \S\ref{sec:exotic}.
To degenerate from $V_{12,2,9}$, we degenerate
the quadric to $uv$ and then choose an elimination term order for
$u,v,w$. The resulting initial ideal is the Stanley-Reisner ideal for
$T_8'*\Delta_0$. 

Finally, consider the equations for $V_{12,3}$. The variables $x_{ijk}$ correspond to
the vertices $(i,j,k)$ of a cube in $\QQ^3$. The equations of
Proposition \ref{prop:caneq} correspond to the affine relations
between these lattice points, see figure \ref{fig:cube}(a). 
The first
six equations correspond to intersecting diagonals on the six faces of
the cube and  the last set of equations corresponds to the four diagonals
intersecting in the middle of the cube. We choose any term order
which, for the first six equations, selects monomials corresponding to
diagonals which form two non-intersecting triangles, see figure
\ref{fig:cube}(b). Degenerating the quadric to the product of the two
vertices not lying on these triangles and taking the initial ideal
gives the desired degeneration.
\end{proof}

\begin{figure}
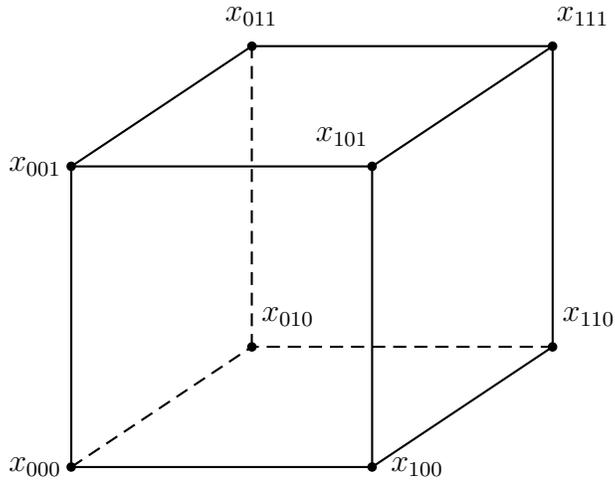
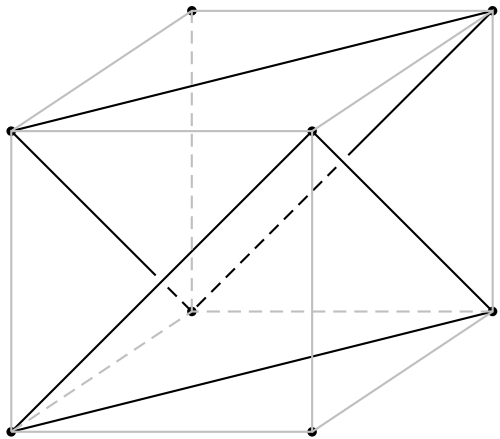
 \subfigure[Monomials on the cube]{\cubeone}
\subfigure[A choice of term order]{\cubetwo}
	\caption{Relations coming from a cube}\label{fig:cube}
\end{figure}

\section{Toric Fano Threefolds of Degree Ten}\label{sec:degten}
We now determine the position of toric Fano threefolds of degree $10$ with at most Gorenstein singularities in $\hilb_{10}$.
Let $\tftv$ denote the set of all  toric Fano
threefolds of degree ten with at most Gorenstein singularities. There are exactly 
54 of these, see
\cite{grdb} and \cite{kreuzer}. 
\begin{thm}\label{thm:re10}
	Consider $X\in\tftv$. Then the point $[X]$ lies exactly on the components of $\hilb_{10}$ as recorded in Table \ref{table:toricten}, where we refer to $X$ by its number in \cite{grdb}.
In particular, $X$ always admits an embedded smoothing to a smooth Fano threefold.
\end{thm}
\begin{proof}
	The theorem is proved using case by case computer computation.
	First we use {\small\verb+TOPCOM+} to check whether the moment polytope of $X$ has a regular unimodular triangulation of the form $T_7*\Delta_0$; this is true exactly for those $X\in\tftv$ not listed in Table \ref{table:toricten}. These $X$ are therefore unobstructed and $[X]$ is a smooth point on $B_{85}$, the component corresponding to $V_{10}$.

	This leaves twelve exceptional cases. In the remainder of this section, we study each of these cases individually by applying the general strategy outlined in \S\ref{sec:def2}. 

\numex{Number 275510}{We describe the five steps of \S\ref{sec:def2} explicitly for $X$ being the toric variety number 275510.
$X$ sits on a scroll of type $(2,2,0,0)$ and is cut out by the $2\times 2$ minors of
\begin{align*}
	M=\left(\begin{array}{c c c c}
	x_0&x_1&y_0&y_1\\
	x_1&x_2&y_1&y_2
	\end{array}\right)
\end{align*}
together with $f_0=x_0^2x_2-y_0z_1z_2$ and the two other cubics $f_1$ and $f_2$ obtained by rolling factors, see the example in \S\ref{sec:rolling}.
The dimension of $T_X^1$ is $27$, and the dimension of $T_X^2$ equals $4$. The space $T_X^1$ can be decomposed into the direct sum of a $24$-dimensional space $T_{\mathrm{roll}}^1$ consisting of perturbations of the cubics $f_i$, a two-dimensional space $T_{\mathrm{scroll}}^1$ generated by the perturbations 
\begin{align*}
	\left(\begin{array}{c c c c}
	x_0&x_1-t_1z_1-t_2z_2&y_0&y_1\\
	x_1&x_2&y_1&y_2
	\end{array}\right)
\end{align*}
of $M$, and a one-dimensional space generated by the non-scrollar perturbation
\begin{align*}
x_0x_2-x_1^2-t_0z_1z_2\\
x_0y_1-x_1y_0\\
x_0y_2-x_1y_1\\
x_1y_1-x_2y_0+t_0x_0x_2\\
x_1y_2-x_2y_1+t_0x_1x_2\\
y_0y_2-y_1y_1
\end{align*}
of the minors of $M$, keeping the $f_i$ constant.

Consider the perturbations of the $f_i$ induced via rolling factors by the perturbation
$$
f_1-s_1x_0z_1^2-s_2x_0z_2^2
$$
of $f_1$. This may be extended to a basis of $T_{\mathrm{roll}}^1$ such that the obstruction equations are 
\begin{align*}
{t}_{1} s_1,\qquad
      {t}_{2} s_2,\qquad
      {t}_{0} {s}_{1},\qquad
      {t}_{0} {s}_{2}.
\end{align*}
This decomposes into the four components $Z_1=V(s_1,s_2)$, $Z_2=V(t_0,t_1,t_2)$, $Z_3=V(t_0,t_1,s_2)$, and $Z_4=V(t_0,s_1,t_2)$. Since the tangent space dimension $h^0(\N_{X/\PP^7})=87$, these cut out schemes of respective dimensions $85$, $84$, $84$, and $84$ in the tangent space of the local Hilbert scheme.

To see that $Z_1$ is a component of the tangent cone of $\hilb_{10}$ at $[X]$, consider the one-parameter deformation $\X\to\Aff^1$ given by the parameter $t_0$. It is straightforward to check that this lifts to higher order with no further perturbations. By construction, the tangent direction of this deformation lies only in the component $Z_1$. 
For $t_0\neq 0$, the resulting ideal is generated by six quadrics, and is easily seen to have a Gr\"obner degeneration to the ideal of $\PP(T_7*\Delta_0)$.  Thus, $Z_1$ must be an $85$-dimensional component of the tangent cone of $\hilb_{10}$ corresponding to the component $B_{85}$.

To see that $Z_2$ is also a component of the tangent cone of $\hilb_{10}$ at $[X]$, 
we may consider a general linear perturbation of $f_1$ (subject to the condition that its factors may be rolled). This defines a general element of the component $B_{84}^\dagger$, since the perturbation still lies on a scroll of type $(2,2,0,0)$.

Finally, we see that $Z_3$ and $Z_4$ are also components of the tangent cone of $\hilb_{10}$ at $[X]$, both corresponding to $B_{84}$. Consider for example the perturbation
\begin{align*}
	\left(\begin{array}{c c c c}
	x_0&x_1-tz_1&y_0&y_1\\
	x_1&x_2&y_1-t^2z_2&y_2
	\end{array}\right)
\end{align*}
along with $f_1-tx_0z_2^2$ and the corresponding rolling factors perturbations. The tangent direction of this perturbation is only contained in $Z_3$; for $t\neq 0$ the fiber is contained in a scroll of type $(1,1,1,1)$ and thus can lie only on the component $B_{84}$. For the case of $Z_4$, a similar perturbation can be made after interchanging $z_1$ and $z_2$.
}

We now provide brief sketches of the remaining cases.

\numex{Number 437961}{$X$ lies in a scroll of type $(1,1,1,1)$ and is cut out by $f_0=x_0y_0z_0-w_0^2w_1$ and the cubics $f_1$ and $f_2$ obtained by rolling factors twice. Thus, $V_{10}'$ degenerates to $X$, so $[X]$ lies on $B_{84}$. A calculation shows that $T_X^2=0$, so $[X]$ is a smooth point of $\hilb_{10}$.
}

\numex{Numbers 86711, 98325, 433633, 439399}{The tangent cone at $X$ has two components, of dimensions $84$ and $85$, cut out by the lowest order terms of the obstruction equations. We can deform onto the $85$-dimensional component, and then degenerate to $\PP(T_7*\Delta_0)$, showing that the $85$-dimensional component is the smoothing component $B_{85}$. 
On the other hand, inspection of the equations of $X$ shows that it is a divisor of type $3D-2F$ on a scroll of type $(2,1,1,0)$. There is an obvious scrollar deformation to a divisor of type $3D-2F$ on a scroll of type $(1,1,1,1)$, which smoothes to $V_{10}'$. Hence, $[X]$ must lie on $B_{84}$.}

\numex{Numbers	522075, 523456, 547399}{Obstruction equations predict that the tangent cone at $X$ has two components, of dimensions $88$ and $85$. We can deform onto the $85$-dimensional component, and then degenerate to $\PP(T_7*\Delta_0)$, showing that the $85$-dimensional component is the smoothing component $B_{85}$. On the other hand, the variety $X$ is cut out by equations of the type for members of $B_{88}^\dagger$ except with degenerate quadrics $g_0,g_1,g_2$. We may conclude that $[X]$ lies on $B_{88}^\dagger$.}

\numex{Numbers 283519, 521212, 522702}{These cases are completely analogous to number 275510 above: all lie on scrolls of type $(2,2,0,0)$. Obstruction equations predict that the tangent cone at $X$ has four components, of dimensions $85$, $84$, $84$, and $84$. Deforming onto the first of these components (with a non-scrollar deformation), we can degenerate to $\PP(T_7*\Delta_0)$, showing that $[X]$ lies on $B_{85}$.  The second component consists only of pure rolling factor deformations, thus corresponding to $B_{84}^\dagger$. The third and fourth components both involve scrollar deformations to a scroll of type $(1,1,1,1)$ and both correspond to $B_{84}$.}
\end{proof}

\section{The Hilbert Scheme $\hilb_{12}$ at $[\PP(T_8'*\Delta_0)]$}\label{sec:bipyramid}
In this section we will study the local structure of $\hilb_{12}$ at the point $[\PP(T_8'*\Delta_0)]$.
We will use this in \S\ref{sec:toric} to help locate the elements of $\tft_{12}$ in $\hilb_{12}$.

Note that $T_8'$ is the join of the boundary $\partial \Delta_1$ of a one-simplex (i.e. two points) with the boundary of a hexagon.
Let $X_\bp=\PP(T_8'*\Delta_0)$.\footnote{The subscript $\bp$ refers to the fact that $T_8'$ is the {\bf b}i{\bf p}yramid over a hexagon.}
We identify the vertices of the hexagon with variables $x_1,\ldots,x_6$ ordered cyclically, the vertices of $\partial \Delta_1$ with variables $y_1,y_2$, and the vertex of $\Delta_0$ with the variable $y_0$. Then $X_\bp$ is cut out by the quadrics $x_{i-1}x_{i+1}$ for $i=1,\ldots,6$  $x_ix_{i+3}$ for $i=1,2,3$, and $y_1y_2$, where all indices are taken modulo six.

We now describe the space $T_{X_\bp}^1$ of first-order deformations of $X_\bp$. Consider the $24$ deformation parameters
$s_i$ and $t_{i,j}$ for $1\leq i \leq 6$ and $j=0,1,2$. We have further $19$ deformation parameters
$a_i,b_i,c_j$ for $1\leq i \leq 6$ and $0\leq j \leq 6$. These $33$ parameters will give us a basis of $T_{X_\bp}^1$. 

To consolidate the presentation, we will write down perturbations of our equations which already include higher order perturbations, since we shall be considering families over the versal base space components. Let $p(z)$ be a power series solution of the functional equation 
$$
zp(z)^4=p(z)+1
$$
and set $f=p(s_1\cdots s_6)$, $e=f/(f+2)$. For $i=1,\ldots,6$, set $t_i=\sum_{j=0}^2t_{i,j}y_j$. We consider the perturbations
\begin{multline}\label{eq:p1}
	x_{i-1}x_{i+1}+(t_i+s_ix_i)x_i\\
+ s_{i+3}(e^2t_{i-2}t_{i+2}+efs_{i+2}t_{i-2}x_{i+2} +
eft_{i+2}s_{i-2}x_{i-2})\\
-s_{i-2}s_{i+2}(et_{i+3}+fs_{i+3}x_{i+3})^2 \\
+ e^2f^2s_{i-2}s_{i-1}s_{i+1}s_{i+2}s_{i+3}t_{i}^2
\end{multline} for $i = 1, \ldots ,6$,
\begin{multline}\label{eq:p2}
	x_ix_{i+3} + et_{i+1}t_{i+2} +
et_{i+2}s_{i+1}x_{i+1} + et_{i+1}s_{i+2}x_{i+2} +
fs_{i+1}s_{i+2}x_{i+1}x_{i+2}\\
 + et_{i-2}s_{i-1}x_{i-1} + et_{i-1}s_{i-2}x_{i-2} +
fs_{i-1}s_{i-2}x_{i-1}x_{i-2}\\
 - e^2f^2s_{i-2}s_{i-1}s_{i+1}s_{i+2}t_it_{i+3}
\end{multline}
for $i=1,2,3$,
and 
\begin{align}\label{eq:p3}
	y_1y_2+c_0y_0^2+\sum_{i=1}^6 (a_ix_i+b_ix_{i+1}+c_iy_0)x_i.
\end{align}
As it stands, the family defined by these perturbations is only flat if considered up to first order. However, we shall see that it becomes flat if we restrict to two of the base space components.
A calculation with {\small\verb+Macaulay2+} or using \cite[Theorem 13]{altmann:04a} shows that with respect to these perturbations, the above deformation parameters form a basis for $T^1_{X_\bp}$.

\begin{thm}\label{thm:tc}
	The tangent cone $\TC_{[\PP(T_8'*\Delta_0)]}\hilb_{12}$ is cut out by the fifteen quadrics 
	\begin{align}		t_{i+1,j}t_{i+2,j}-t_{i-1,j}t_{i-2,j}\qquad &i\in\{1,2,3\},\quad j\in\{0,1,2\}\label{eqn:tc1}
\\
		t_{i+1,j}t_{i+2,0}+t_{i+1,0}t_{i+2,j}-t_{i-1,j}t_{i-2,0}-t_{i-1,0}t_{i-2,j}\qquad &i\in\{1,2,3\},\quad j\in\{1,2\}.\label{eqn:tc2}
\end{align}
It decomposes into four irreducible components $Z_{97}$, $Z_{99}$, $Z_{98}^1$, $Z_{98}^2$ of respective dimensions $97$, $99$, $98$, and $98$. $Z_{97}$ is cut out by the $2\times 2$ minors of
\begin{align*}	\left(\begin{array}{c c c c c c}
		t_{1,0} & t_{1,1} & t_{1,2} & t_{4,0} &t_{4,1}& t_{4,2}\\
		t_{3,0} & t_{3,1} & t_{3,2} & t_{6,0} &t_{6,1}& t_{6,2}\\
		t_{5,0} & t_{5,1} & t_{5,2} & t_{2,0} &t_{2,1}& t_{2,2}
	\end{array}\right)	
\end{align*}
and corresponds to the component $B_{97}$.
$Z_{99}$ is cut out by the $2\times 2$ minors of
\begin{align*}
	\left(\begin{array}{c c c c c c c c c}
		t_{1,0} & t_{1,1} & t_{1,2} &	t_{3,0} & t_{3,1} & t_{3,2}&
		t_{5,0} & t_{5,1} & t_{5,2}\\
		 t_{4,0} &t_{4,1}& t_{4,2}&
		 t_{6,0} &t_{6,1}& t_{6,2}
		& t_{2,0} &t_{2,1}& t_{2,2}
	\end{array}\right)
\end{align*}
and corresponds to the component $B_{99}$. Finally, both components $Z_{98}^1$, $Z_{98}^2$ correspond to $B_{98}$ and for $k,l=1,2$ with $k\neq l$, $Z_{98}^k$ is cut out by the thirty quadrics
\begin{align*}
	t_{i+1,j}t_{i+2,j}-t_{i-1,j}t_{i-2,j}\qquad & i\in\{1,2,3\},\quad j\in\{0,1,2\}\\
	t_{i+1,k}t_{i-1,0}-t_{i+1,0}t_{i-1,k}\qquad & i\in\{1,2,3,4,5,6\}\\
	t_{i+1,k}t_{i+2,0}-t_{i-1,k}t_{i-2,0}\qquad & i\in\{1,2,3,4,5,6\}\\
	t_{i,l}t_{i+1,0}-t_{i+3,l}t_{i+1,0}\qquad & i\in\{1,2,3\}\\
	t_{i,l}t_{i-1,0}-t_{i+3,l}t_{i-1,0}\qquad & i\in\{1,2,3\}\\
	t_{i,l}t_{i+3,0}-t_{i,0}t_{i+3,l}\qquad & i\in\{1,2,3\}.
\end{align*}
\end{thm}
\begin{proof}
	Using {\small\verb+Versal Deformations+}, we can calculate that the lowest order terms of the obstruction equations are exactly the quadrics in \eqref{eqn:tc1} and \eqref{eqn:tc2}. The tangent cone is certainly contained in the subscheme $Z$ cut out by these equations. Using primary decomposition in {\small\verb+Macaulay2+}, we see that $Z$ decomposes into the four components $Z_{97}$, $Z_{99}$, $Z_{98}^1$, and $Z_{98}^2$, which have the stated dimension.

	Now, by Proposition \ref{prop:srdegen}, we know that $[\PP(\T_8'*\Delta_0)]$ lies on $B_{97}$, $B_{98}$, and $B_{99}$. Thus, the tangent cone at this point must have components of dimensions $97$, $98$, and $99$. $Z_{98}^1$, and $Z_{98}^2$ are indistinguishable modulo a $\ZZ_2$ symmetry, so we can conclude that the lowest order terms of the obstruction equations actually cut out the tangent cone.
\end{proof}
If we ignore the component $B_{98}$, we can even say more about the local structure of $\hilb_{12}$ at
the point $[\PP(T_8'*\Delta_0)]$.
\begin{thm}\label{thm:uf}
	In a  formally local neighborhood of $[\PP(T_8'*\Delta)]\in\hilb_{12}$, the components $B_{97}$ and $B_{99}$ are respectively cut out by the equations for $Z_{97}$ and $Z_{99}$. Over these components, a universal family $\mathcal{U}$ is given by the perturbations \eqref{eq:p1}, \eqref{eq:p2}, and \eqref{eq:p3} after adding linear changes of coordinates to account for trivial deformations.
\end{thm}
\begin{proof}
	We claim that the family defined by \eqref{eq:p1}, \eqref{eq:p2}, and \eqref{eq:p3} is flat if we impose the equations for either $Z_{97}$ or $Z_{99}$. Indeed, by \cite[Proposition 6.6]{altmann:09a}, the family defined by  \eqref{eq:p1} and \eqref{eq:p2} is flat if we require the vanishing of the $2\times 2$ minors of 
	\begin{align*}
		\left(\begin{array}{c c c}
t_1&t_3&t_5\\
t_4&t_6&t_2
	\end{array}\right).
\end{align*}
The equations cutting out $Z_{97}$ and $Z_{99}$ are two different ways of satisfying this condition.
When we add the equation \eqref{eq:p3}, the additional relations are simply Koszul relations and can be lifted trivially.

Since this family spans the vector space of first order deformations, the statement of the theorem follows.
\end{proof}
\section{Toric Fano Threefolds of Degree Twelve}\label{sec:toric}
We now determine the position of toric Fano threefolds of degree $12$ with at most Gorenstein singularities in $\hilb_{12}$.
Let $\tftw$ be the set of all Gorenstein toric Fano
threefolds of degree twelve. There are exactly 135 of these, see
\cite{grdb} and \cite{kreuzer}.
\begin{thm}\label{thm:re12}
	Consider $X\in\tftw$. Then the point $[X]$ lies exactly on the components of $\hilb_{12}$ as recorded in Table \ref{table:torictwelve}, where we refer to $X$ by its number in \cite{grdb}.
	In particular, $X$ always admits an embedded smoothing to a smooth Fano threefold. Furthermore, each component $B$ of $\hilb_{12}$ is smooth at $[X]$, unless $[X]$ lies on $B_{97}$, $B_{98}$ and $B_{99}$ (and possibly $B_{99}^\dagger$) and $B=B_{98}$, or if $X$ is number 544886 and $B=B_{96}$.
\end{thm}
\begin{proof}
The theorem is proved using case by case computer computation.
We use {\small\verb+TOPCOM+} to partition $\tftw$ into three subsets:
\begin{enumerate}
	\item Those $X\in\tftw$ whose moment polytope has a regular unimodular triangulation of the form $T_8*\Delta_0$;
	\item Those $X\in\tftw$ whose moment polytope has a regular unimodular triangulation of the form $T_8'*\Delta_0$ but not of the form $T_8*\Delta_0$;
	\item Those $X\in\tftw$ whose moment polytope does not have a regular unimodular triangulation of the form $T_8*\Delta_0$ or $T_8'*\Delta_0$.
\end{enumerate}

If $X$ belongs to set (i), then $X$ is unobstructed and $[X]$ is a smooth point on $B_{98}$, the component corresponding to $V_{12}$. This covers all $X\in\tftw$ not listed in Table \ref{table:torictwelve}. 

The set (ii) is a subset of all those $X\in\tftw$ which are listed in Table \ref{table:torictwelve} as lying on $B_{97}$ or $B_{99}$. It excludes exactly those lying on $B_{96}$ or $B_{99}^\dagger$ and numbers 321879 and 524375.
Since such $X$ has a moment polytope with a regular unimodular triangulation of the form $T_8'*\Delta_0$, it follows from Theorem \ref{thm:tc} that the only possible components of $\hilb_{12}$ on which $[X]$ can lie are $B_{97}$, $B_{98}$, or $B_{99}$. Now, using the triangulation to $T_8'*\Delta_0$, we can explicitly find a curve in $\hilb_{12}$ passing through $[X]$ and $[X_{\bp}]$. Using the local universal family $\mathcal{U}$  over $B_{97}$ and $B_{99}$ from Theorem \ref{thm:uf}, it turns out that $X$ always appears as fiber of $\mathcal{U}$ defined by polynomials (instead of power series). This allows us to determine exactly which of the components $B_{97}$ and $B_{99}$ $[X]$ lies on, and using the local equations for these components, whether that component is smooth at $[X]$. Computations show that we in fact always have smoothness in these cases.

For such cases, it remains to be seen if $[X]$ also lies on $B_{98}$ (and whether that is a smooth point on the component). Since $[X]$ is a smooth point on $B_{97}$ and/or $B_{99}$, the components of the tangent cone corresponding to $B_{97}$ and/or $B_{99}$ must be cut by equations whose lowest order terms are linear. Suppose that $[X]$ does not lie on $B_{98}$. 
Then the ideal of the tangent cone will be generated by equations with lowest order term quadratic if $[X]$ lies on both $B_{97}$ and $B_{99}$, and by equations with lowest order term linear otherwise. Thus, in such cases, the tangent cone is cut out by the lowest order obstruction equations.

We may now proceed as follows. First, we use {\small\verb+Versal Deformations+} to compute the lowest order terms for the obstruction equations of $X$, and decompose the scheme $Z$ cut out by these equations into irreducible components. By the above argument, if this decomposition includes anything but smooth components of dimensions $97$ and $99$, $[X]$ must lie on $B_{98}$. The claim regarding the smoothness of $B_{98}$ follows in the appropriate cases from the fact that, in these cases, the additional components of $Z$ a posteriori consist of a single smooth $98$-dimensional component.

\begin{ex}Let $X$ be the toric variety number 5953. Then the ideal of $X$ is generated by the ten binomials 
	\begin{align*}
x_2x_6-y_0x_1,\qquad x_1x_3-y_0x_2,\qquad x_2x_4-y_0x_3,\\
x_3x_5-y_0x_4,\qquad x_4x_6-y_0x_5,\qquad x_1x_5-y_0x_6,\\
x_1x_4-y_0^2,\ \ \ \qquad x_2x_5-y_0^2,\ \ \ \qquad x_3x_6-y_0^2,\\
y_1y_2-y_0x_1.	
\end{align*}
Note that the lead monomials of these binomials are just the generators of the Stanley-Reisner ideal of $T_8'*\Delta$. Using the universal family from Theorem \ref{thm:uf}, the point $[X]$ locally has coordinates $t_{i,0}=-1$ for $i=1,\ldots,6$ and $c_1=-1$, with all other coordinates vanishing. The $3\times 6$ and $2\times 9$ matrices appearing in Theorem \ref{thm:tc} both have rank $1$ when evaluated at this point. 
Thus, $[X]$ is a smooth point on both $B_{97}$ and $B_{99}$.

The lowest order terms of the obstruction equations for $X$ have the form
$t_1t_3$, $t_1t_4$, $t_2t_5$, $t_2t_6$, which decomposes into the components $V(t_1,t_2)$ of dimension $98$, $V(t_1,t_5,t_6)$ and $V(t_2,t_3,t_4)$ of dimensions $98$, and $V(t_3,t_4,t_5,t_6)$ of dimension $97$. Thus, $[X]$ must lie on $B_{98}$ as well.
\end{ex}

The above two techniques deal with almost all $X\in\tftw$; for the set (iii) we are left with 10 exceptional cases which we deal with as described in \S\ref{sec:def2}. Once we have identified the number and dimension of components of the tangent cone at $[X]$ using obstruction calculus and one-parameter deformations, we still need to match these components to $B_{96}$, $B_{97}$, $B_{98}$, $B_{99}$, and $B_{99}^\dagger$, that is, step (v) from \S\ref{sec:def2}. For $B_{96}$ this is done by finding explicit degenerations of $V_{12,2,6}$  by using the rolling factors description in Proposition \ref{prop:caneq}. Likewise, for $B_{99}^\dagger$ we also use the rolling factors format to find explicit degenerations. For $B_{97}$, $B_{98}$, $B_{99}$, we find explicit degenerations to $\PP(T_8*\Delta_0)$ and/or $\PP(T_8'*\Delta_0)$.
The cases {\bf 146786}, {\bf 444999}, {\bf 544855}, and {\bf 544887} are dealt with in a straightforward manner. The remaining six cases are all more difficult, since their tangent cones appear to contain embedded components. 
In the remainder of this section, we study each of these cases individually by applying the general strategy outlined in \S\ref{sec:def2}.
\numex{Number 147467}{
Let $X$ be the toric variety number 147467. Then $X$ is a divisor of type $4D-3F$ on a scroll of type $(2,2,1,0)$. If the scroll is given by the maximal minors of 
$$
M=\left(\begin{array}{c c c c c}
x_0&x_1&y_0&y_1&z_0\\
x_1&x_2&y_1&y_2&z_1\\
\end{array}
\right)
$$
then $X$ is cut out by rolling the cubic $f_0=x_0z_0w-y_0^2y_1$ three times to $f_1,f_2,f_3$.

The dimension of $T_X^1$ is $22$, that of $H^0(\N_X)$ is $99$, and the dimension of $T_X^2$ equals $4$. The space $T_X^1$ can be decomposed into the direct sum of a $20$-dimensional space $T_{\mathrm{roll}}^1$ consisting of pure rolling factors perturbations of the cubics $f_i$, a one-dimensional space $T_{\mathrm{scroll}}^1$ generated by the perturbation 
$$
\left(\begin{array}{c c c c c}
x_0&x_1&y_0&y_1&z_0\\
x_1&x_2&y_1-t_1w&y_2&z_1\\
\end{array}
\right)
$$
of $M$, and a one-dimensional space generated by a certain non-scrollar perturbation
of the minors of $M$, keeping the $f_i$ constant.
A basis of $T_X^1$ may be chosen such that the lowest order terms
of the obstruction equations include $t_1^2t_3$, $t_1^2t_4$, and $t_1t_2$, where $t_1$ is as above, $t_2$ is a parameter for the non-scrollar perturbation, and $t_3,t_4$ are parameters for pure rolling factors deformations.\footnote{There is one additional element of $T^2$ which, at least up to order $8$ does not contribute an obstruction equation.}

The scheme $Z$ in the tangent space $\hilb_{10}$ cut out by these three monomials decomposes into components $Z_{98}=V(t_1)$ of dimension $98$, $V(t_2,t_3,t_4)$ of dimension $96$, and an embedded component $Z_{97}=V(t_2,t_1^2)$ of dimension $97$. The deformation in the $t_1$ direction described above deforms the scroll to one of type $(2,1,1,1)$, which gives a deformation to $V_{12,2,6}$. Thus, $Z_{96}$ is a component of the tangent cone corresponding to $B_{96}$. Similarly, a deformation in the $t_2$ direction takes us to a scheme $X'$ which degenerates to $\PP(T_8*\Delta_0)$, so  $Z_{98}$ corresponds to $B_{98}$. 

We need to check that the  component $Z_{97}$ does not correspond to $B_{97}$. Since $Z_{97}$ is embedded in $Z_{98}$, there is no obvious way to deform onto a $97$-dimensional component. In fact, by the discussion below, we will see that $Z_{97}$ cannot correspond to a non-embedded $97$-dimensional component of $\hilb_{12}$.
}

Let $S\subset\hilb_{12}$ be the closure of the set of all points corresponding to divisors of type $4D-3F$ on a scroll of type $(2,2,1,0)$. From the above example, it follows that $\dim S = 97$.  
\begin{lemma}\label{lemma:scroll}
	If $\eta$ is a general point of $S$, then $\dim T_{\eta}\hilb_{12}=99$.
\end{lemma}
\begin{proof}
	Let $Y$ correspond to $\eta$, i.e. $[Y]=\eta$.
	By the above example, $\dim T_{\eta}\hilb_{12}\leq 99$. Now, consider the subscheme $\mathcal{Y}$ of $\PP^{16}$ defined by the maximal minors of the matrix $M$ above, along with cubics $f_0,\ldots,f_3$ obtained by rolling factors, where 
	$$
f_0=z_0^3+s_1x_0^2+s_2x_0x_1+s_3x_0x_2+s_4x_0y_0+s_5x_0y_1+s_6y_0y_0+s_7y_0y_1+s_8y_0y_2.
	$$
	Here, $x_0,\ldots,w,s_1,\ldots,s_8$ are coordinates on $\PP^{16}$. Then $Y$ is codimension $8$ linear section of $\mathcal{Y}$. A computer calculation shows that $T_\mathcal{Y}^1$ is two-dimensional. Thus, there are two deformation directions in $T_Y^1$ which are not of pure rolling factor type. Hence, $\dim T_{\eta}\hilb_{12}\geq \dim S+2=99$. 	
\end{proof}

\begin{prop}Let $X$ be the toric Fano threefold 147467. The component $Z_{97}$ from the example above cannot correspond to a $97$-dimensional component of $\hilb_{12}$ which is not embedded.
\end{prop}
\begin{proof}
	Suppose $Z_{97}$ corresponds to a non-embedded 97-dimensional
	component $B\subset\hilb_{12}$. The scheme $Z_{98}$ is smooth, so it follows that $[X]$ is a smooth point on $B_{98}$. 
	Now, a general point $\eta\in S$ does not lie on $B_{96}$, and thus must lie on either $B$ or $B_{98}$ or both. But $\dim T_{\eta} \hilb_{12}=99$ by the above lemma, and  if $\eta\in B$, $\dim T_{\eta} B\leq 98$, and if $\eta\in B_{98}$, $\dim T_{\eta} B_{98}=98$. Thus, $\eta\in B\cap B_{98}$, so we have  $S\subset B\cap B_{98}$. But $B_\mathrm{red}=S$ for dimension reasons, which means $B$ is embedded in $B_{98}$, a contradiction.
\end{proof}

This concludes the discussion of the case 147467. The cases {\bf 446913} and {\bf 544886} are almost identical, also being divisors on a scroll of type $(2,2,1,0)$. For the remaining three cases (321879, 524375, and 547426) the lowest order terms of the obstruction equations give rise to embedded components of dimensions $97$ and less. We must show that these cases do not lie on $B_{96}$.

\numex{Number 524375}{
Let $X$ be the toric variety number 524375. Then $X$ is a divisor of type $4D-3F$ on a scroll $S$ of type $(3,2,0,0)$. If the scroll is given by the maximal minors of 
$$
M=\left(\begin{array}{c c c c c}
x_0&x_1&x_2&y_0&y_1\\
x_1&x_2&x_3&y_1&y_2\\
\end{array}
\right)
$$
then $X$ is cut out by rolling the cubic $f_0=x_0zw-y_0^3$ three times to $f_1,f_2,f_3$.

Now, suppose that $X$ has an embedded smoothing to $V_{12,2,6}$, which is a divisor on a scroll of type $(2,1,1,1)$. Then the deformation of $X$ corresponds to a deformation of the scroll $S$ to a scroll $S'$ of type $(2,1,1,1)$. Indeed the perturbations of the quadrics cutting out $S$ must cut out a scroll of type $(2,1,1,1)$. Now, a versal deformation of $S$ is given by linear perturbations of the $x_1$, $x_2$, and $y_1$ entries of the top row of $M$. In order to deform to a scroll of type $(2,1,1,1)$, either the $x_1$ or $x_2$ entry must be perturbed nontrivially. But this makes it impossible to roll $f_0$ (or a perturbation thereof) three times. Thus, $X$ does not deform to $V_{12,2,6}$ and $[X]$ does not lie on $B_{96}$.}

Similar arguments may be made for {\bf 321879} and {\bf 547426} which lie on scrolls of type $(3,2,0,0)$ and $(4,1,0,0)$, respectively. 
This completes the proof of the theorem.
\end{proof}

\begin{remark}
  Although we have only studied the Hilbert schemes of Fano threefolds of degree at most twelve, our techniques may be applied to study the case of higher degree Fanos as well. In particular, Theorem \ref{thm:nice} applies in the cases of degree $14$ and $16$, and a modified version of the theorem applies in the case of degree $18$, cf.~final remark of \cite[\S 3]{paperone}. 
  This already allows us to determine the location in the Hilbert scheme of a significant number of degree $14$, $16$, and $18$ Gorenstein toric Fano threefolds. For example, a computation shows that of the $207$ varieties in $\tft_{14}$, $43$ degenerate to $\PP(T_9*\Delta_0)$, and hence are smooth points on the Hilbert scheme component of degree $14$ rank one index one smooth Fano threefolds.

  The main obstruction to a complete result for higher degrees is the increased appearance of toric varieties with complicated obstruction equations. We often find ourselves in a situation similar to that of number 147467 (cf.~the proof of Theorem \ref{thm:re12}), where the lowest order terms of the obstruction equations cut out a non-reduced scheme. In order to appropriately adapt our general strategy of \S\ref{sec:def2}, we would need in each case some further structural result such as Lemma \ref{lemma:scroll}.
\end{remark}

\bibliographystyle{amsalpha}
\bibliography{deg12}

\end{document}

%% file: figures12.tex
\newcommand{\tritwelve}{
\psset{unit=.45cm}
\begin{pspicture}(-5,-5)(5,5)
\psline(-4,-4)(4,4)
\psline(-4,0)(4,0)
\psline(0,4)(0,-4)
\psline(-2,-2)(0,-2)(2,0)(2,2)(0,2)(-2,0)(-2,-2)
\psdots(0,0)(-2,-2)(0,-2)(-2,0)(2,0)(0,2)(2,2)
\psdots[dotstyle=o](-4,-4)(0,-4)(-4,0)(4,0)(0,4)(4,4)
\end{pspicture}
}

\newcommand{\cubeone}{
\psset{unit=.4cm}
\begin{pspicture}(-2,-2)(18,16)
\psdots(0.00,0.00)(10.00,0.00)(0.00,10.00)(10.00,10.00)(6.00,4.00)(16.00,4.00)(6.00,14.00)(16.00,14.00)

\psline(0.00,0.00)(10.00,0.00)(10.00,10.00)(0,10)(0,0)
\psline(10,0)(16.00,4.00)(16.00,14.00)(10,10)
\psline(0,10)(6,14)(16,14)
\psline[linestyle=dashed](0,0)(6,4)(16,4)
\psline[linestyle=dashed](6,4)(6,14)
\rput(-1.2,0){$x_{000}$}
\rput(11.5,0){$x_{100}$}
\rput(-1.2,10){$x_{001}$}
\rput(7.2,5){$x_{010}$}
\rput(17.2,5){$x_{110}$}
\rput(17,15){$x_{111}$}
\rput(6,15){$x_{011}$}
\rput(9,11){$x_{101}$}
\end{pspicture}
}

\newcommand{\cubetwo}{
\psset{unit=.4cm}
\begin{pspicture}(-2,-2)(18,16)
\psdots(0.00,0.00)(10.00,0.00)(0.00,10.00)(10.00,10.00)(6.00,4.00)(16.00,4.00)(6.00,14.00)(16.00,14.00)
\psline(0,0)(10,10)
\psline(0,10)(16,14)
\psline(0,0)(16,4)(10,10)
\psline(0,10)(4.8,5.2)
\psline(11.2,9.2)(16,14)
\psline[linestyle=dashed](5.2,4.8)(6,4)(10.8,8.8)

\psset{linecolor=lightgray}
\psline(0.00,0.00)(10.00,0.00)(10.00,10.00)(0,10)(0,0)
\psline(10,0)(16.00,4.00)(16.00,14.00)(10,10)
\psline(0,10)(6,14)(16,14)
\psline[linestyle=dashed](0,0)(6,4)(16,4)
\psline[linestyle=dashed](6,4)(6,14)
\end{pspicture}
}

\newcommand{\torictentable}{
\begin{tabular}{|@{$\quad$}p{5cm}@{$\quad$}|@{$\quad$}  p{9cm}@{$\quad$}|}
\hline
& \# in the Graded Ring Database \cite{grdb}\\
\hline
&\\
Points on $B_{84}$&
437961
\\
&\\
\hline
&\\
Points on $B_{84}\cap B_{85}$ &
86711, 98325, 433633, 439399 
\\
&\\
\hline
&\\
Points on $B_{85}\cap B_{88}^\dagger$ &
522075, 523456, 547399
\\
&\\
\hline
&\\
Points on $B_{84}\cap B_{85}\cap B_{84}^\dagger$ &
275510, 283519, 521212, 522702 
\\
&\\
\hline
\multicolumn{2}{|c|}{}\\
\multicolumn{2}{|c|}{Points not listed are smooth points on $B_{85}$.}\\
\multicolumn{2}{|c|}{}\\
\hline

\end{tabular}
}

\newcommand{\torictwelvetable}{
\begin{tabular}{|@{$\quad$}p{5cm}@{$\quad$}|@{$\quad$}  p{9cm}@{$\quad$}|}
\hline
& \# in the Graded Ring Database \cite{grdb}\\
\hline
&\\
Points on $B_{96}$&
146786, 444999, 544855
\\
&\\
\hline
&\\
Points on  $B_{97}$&
525553, 545072
\\
&\\
\hline
&\\
Points on $B_{99}$&
317924,  525745, 545139
\\
&\\
\hline
&\\
Points on $B_{96}\cap B_{98}$&
147467\fooast, 446913\fooast, 544886\fooast
\\
&\\
\hline
&\\
Points on $B_{97}\cap B_{98}$&
95245, 281906, 292940, 439663, 521504
\\
&\\
\hline
&\\
Points on $B_{98}\cap B_{99}$&
263867, 274128, 281846, 439636, 446933, 446982, 522703, 544529, 544753, 573895, 585686
\\
&\\
\hline
&\\
Point on $B_{96}\cap B_{97}\cap B_{98}$&
544887\fooast
\\
&\\
\hline
&\\
Points on $B_{97}\cap B_{98}\cap B_{99}$&
5953, 5954, 29624, 72202, 72684, 87167, 147470, 261714, 264855, 269333, 275527, 283523, 294031, 294043, 306960, 321877, 321879\fooast, 355616, 431397, 431910, 432464, 432661, 434956, 435216, 442762, 452175, 466014, 520330, 520423, 520890, 520931, 522058, 522085, 522683, 524375\fooast, 524380, 524396, 526891, 528560, 530438, 544400, 544417, 544419, 544458, 544530, 544536, 544539, 544697, 544885, 545307, 547388, 547402, 547423, 547460, 566695, 566716, 585890, 585895, 585897, 610803
\\
&\\
\hline
&\\
Point on $B_{97}\cap B_{98}\cap B_{99}\cap B_{99}^\dagger$&
547426\fooast
\\
&\\
\hline
\multicolumn{2}{|c|}{}\\
\multicolumn{2}{|c|}{Points not listed are smooth points on $B_{98}$.}\\
\multicolumn{2}{|c|}{}\\
\multicolumn{2}{|c|}{Points marked with $*$ also lie on other}\\
\multicolumn{2}{|c|}{non-smoothing components (possibly embedded).}\\
\multicolumn{2}{|c|}{}\\
\hline

\end{tabular}
}